\documentclass{amsart}

\usepackage{amssymb}

\theoremstyle{definition}
\newtheorem{definition}{Definition}[section]
\newtheorem{theorem}{Theorem}[section] 
\newtheorem{corollary}[theorem]{Corollary}
\newtheorem{proposition}[theorem]{Proposition}

\title{Banach spaces with polynomial numerical index 1} 

\author[H.~J.~Lee]{Han Ju Lee}
\address{Han Ju Lee\\ Mathematics Department\\
202 Mathematical Sciences Bldg\\ University of Missouri-Columbia\\
MO 65211 USA} \email{hahnju@postech.ac.kr}


\thanks{This work was supported by the Korea Research Foundation
Grant funded by the Korean Government(MOEHRD) (KRF-2006-352-C00003)
and by grant No. R01-2004-000-10055-0 from the Basic Research
Program of the Korea Science \& Engineering Foundation.
}

\begin{document}
\maketitle

\begin{abstract}
We characterize Banach spaces with polynomial numerical index 1 when
they have the Radon-Nikod\'ym property. The holomorphic numerical
index is introduced and the characterization of the Banach space
with holomorphic numerical index 1 is obtained when it has the
Radon-Nikod\'ym property.
\end{abstract}

\section{Introduction} Let $X$ be a Banach space over a real or complex scalar field
$\mathbb{K}$ and $X^*$ its dual space. We denote by $\mathcal{L}(X)$
the Banach space of all bounded linear operators from $X$ to itself
with usual operator norm. We consider the topological subspace
$\Pi(X)=\{ (x, x^*): x^*(x)=1=\|x\|=\|x^*\|\}$ of the product space
$B_X\times B_{X^*}$, equipped with norm and weak-$*$ topology on
$B_X$,  the unit ball of $X$ and $B_{X^*}$ respectively. It is easy
to see that $\Pi(X)$ is a closed subspace of $B_X\times B_{X^*}$.
The {\bf numerical radius} $v(T)$ of a linear operator $T:X\to X$ is
defined by $v(T) = \sup \{ |x^*Tx|: (x, x^*)\in \Pi(X)\}$. We denote
by $n(X)$ the {\bf numerical index} of $X$ defined by $n(X)=\inf \{
v(T): T\in \mathcal{L}(X), \|T\|=1\}$.

Notice that $0\le n(X)\le 1$ and $n(X)\|T\|\le v(T)\le \|T\|$ for
any $T\in \mathcal{L}(X)$. Hence $v(\cdot )$ is equivalent to the
operator norm on $\mathcal{L}(X)$ when $n(X)>0$. For more properties
about the numerical radius and index, see \cite{KMP}. As in
\cite{H}, the notion of numerical radius can be extended to elements
in $C_b(B_X:X)$ of bounded continuous functions from $B_X$ to $X$.
More precisely, for each $f\in C_b(B_X:X)$, $v(f)= \sup\{
|x^*f(x)|:(x, x^*)\in \Pi(X)\}$. Notice that $C_b(B_X:Y)$ is a
Banach space equipped with sup norm $\|f\|=\sup\{ \|f(x)\|:x\in
B_X\}$.

Let $X$ and $Y$ be Banach spaces and $k\ge 1$. A bounded
$k$-homogeneous polynomial $P:X\to Y$ is defined to be
$P(x)=L(x,\dots, x)$ for all $x\in X$, where $L:X\times \cdots
\times X\to Y$ is a continuous $k$-multilinear map. We denote by
$\mathcal{P}({\ }^k X:Y)$ the Banach space of all bounded
$k$-homogenous polynomials from $X$ to $Y$ as a subspace of
$C_b(B_X:Y)$. Following the notations in \cite{CGKM}, the {\bf
$k$-polynomial numerical index} of a Banach space $X$ is defined to
be $n^{(k)}(X)= \inf \{ v(P): P\in \mathcal{P}({\ }^k X:X),
\|P\|=1\}$. It is easy to see that $n^{(k+1)}(X)\le n^{(k)}(X)$ for
any $k\ge 1$.

Let $X$ be a real or complex Banach space and let $K$ be a nonempty
convex subset of $X$. Recall that $x\in K$ is said to be an {\bf
extreme point} of $K$ if whenever $y+z=2x$ for some $y,z\in K$, we
have $x=y=z$. Denote by $ext(K)$ the set of all extreme points of
$K$.

The geometric properties of a Banach space $X$ with $n(X)=1$ have
been studied \cite{LMP,M,MG}. McGregor \cite{MG} gave a geometric
characterization of finite dimensional Banach spaces with numerical
index 1. More precisely, a finite dimensional Banach space $X$ has
numerical index 1 if and only if $|x^*(x)|=1$ for every $x\in
ext(B_X)$ and for every $x^*\in ext(B_{X^*})$. For a Banach space
$X$ with the Radon-Nikod\'ym property, it was shown \cite{LMP,M}
(cf. \cite[Corollary 1]{KMP}) that $n(X)=1$ if and only if
$|x^*(x)|=1$ for every $x^*\in ext(B_{X^*})$ and for every denting
point $x$ of $B_X$.  For the definition of denting point, see
\cite{DU}. In \cite{KMP}, they asked if there are similar
characterizations of the Banach space $X$ with $n^{(k)}(X)=1$ for
each $k\ge 2$. In this paper, we give the partial answer.

We need the following notion which is introduced by Ferrera \cite{F}
for the $k$-homogeneous polynomials.
\begin{definition}
Let $X$ and $Y$ be Banach spaces over $\mathbb{K}$. A nontrivial
function $f\in C_b(B_X: Y)$ is said to {\bf strongly attain its norm
} at $x_0$ if whenever $\lim_n \|f(x_n)\|=\|f\|$ for a sequence
$\{x_n\}$ in $B_X$, it has a subsequence $\{x_{n_k}\}_{k=1}^\infty$
converging to $\alpha x_0$ for some $|\alpha|=1$, $\alpha\in
\mathbb{K}$. Let $H$ be a subspace of $C_b(B_X:Y)$. Denote by
$\tilde\rho H$ the set
\[\tilde\rho H = \{ x_0 : f \mbox{  strongly attains its norm at } x_0 \
\mbox{and}\ \ f\in H\}.\] A nonzero function $f\in C_b(B_X:Y)$ is
said to be a {\bf strong peak function} at $x_0$ if whenever there
is a sequence $\{x_n\}_{n=1}^\infty$ in $B_X$ with $\lim_n
\|f(x_n)\|=\|f\|$, the sequence $\{x_n\}_n$ converges to $x_0$. We
 denote by $\rho H$ the set
\[ \rho H = \{ x_0 : f \mbox{ is a strong peak function in } H \mbox{ at }
x_0\}.\]
\end{definition}
Notice that $f\in C_b(B_X:Y) $ strongly attains its norm at $x$ if
and only if for each $\epsilon>0$, there is $\delta>0$ such that
whenever $\|f(y)\|\ge \|f\|-\delta$ for some $y\in B_X$, we get
$\inf_{\lambda\in S_\mathbb{K}} \|\lambda x - y\|\le \epsilon$.
Notice also that $\tilde\rho\mathcal{P}({\ }^k X:Y)
=\tilde\rho\mathcal{P}({\ }^k X)$ for each $k\ge 1$ if $Y$ is
nontrivial.

For two complex Banach spaces $X$ and $Y$, we are interested in two
 subspaces of $C_b(B_X:Y)$,
\begin{eqnarray*}
A_b(B_X: Y) &=& \{ f\in C_b(B_X: Y):
f\mbox{ is holomorphic on the open unit ball}~{B_X^\circ}\}\\
A_u(B_X: Y) &=& \{ f\in A_b(B_X: Y): f \mbox{ is uniformly
continuous on } B_X\}.
\end{eqnarray*} We denote by $A(B_X:Y)$ either $A_u(B_X:Y)$ or
$A_b(B_X:Y)$. When $Y=\mathbb{C}$, we write $A(B_X)$ instead of
$A(B_X:\mathbb{C})$. By the maximum modulus theorem, it is easy to
see that if $f\in A_b(B_X:Y)$ strongly attains its norm at $x_0$ and
$X$ is nontrivial, then $x_0$ is contained in $S_X$, the unit sphere
 of $X$.

\section{Main results}

Let $X$ and $Y$ be complex Banach spaces. It is shown (see the proof
of \cite[Theorem~4.4]{CLS}) that if $f\in A(B_X: Y)$ strongly
attains its norm at $x_0$. Then, given $\epsilon>0$, there is $g\in
A(B_X:Y)$ with $\|g\|<\epsilon$ and $f+g$ is a strong peak function
at $\alpha x_0$ for some $\alpha\in S_\mathbb{C}$. This implies that
$x_0$ is a strong peak point of $A(B_X:Y)$. So it is easy to see
that $\rho A(B_X)=\rho A(B_X:Y) = \tilde\rho A(B_X:Y) = \tilde\rho
A(B_X)$.  Recall that $x\in B_X$ is said to be a {\bf complex
extreme point} of $B_X$ if $\sup_{0\le \theta\le 2\pi} \|x+
e^{i\theta} y\| \le 1$ for some $y\in X$ implies $y=0$. We denote by
$ext_\mathbb{C}(B_X)$ the set of all complex extreme points of
$B_X$. Notice that every strong peak point of $A(B_X)$ is a complex
extreme point of $B_X$ (see \cite{Gl}). So we have the following.
\begin{proposition}Let $X$ and $Y$ be complex Banach spaces and
$H$ a subspace of $A(B_X: Y)$. Then
\[ \tilde\rho H \subset \rho A(B_X:Y) \subset ext_\mathbb{C}(B_X).\]
\end{proposition}
It is worth while to remark that when $X$ is a nontrivial finite
dimensional complex Banach space, then $\rho A(B_X)=
ext_\mathbb{C}(B_X)$ (see \cite{CHL}).

Let $X$ be a real or complex Banach space and let $K$ be a nonempty
convex subset of $X$. Recall that an element $x$ in $K$ is said to
be a {\bf strongly exposed point} of $K$ if there is nonzero $x^*\in
B_{X^*}$ such that ${\rm Re\ }x^*(x)= \sup \{ {\rm Re\ }x^*(y):y\in
K\}$ and whenever $\lim_n {\rm Re\ } x^*(x_n)= {\rm Re\ }x^*(x)$ for
some sequence $\{x_n\}_{n=1}^\infty$ in $K$, we get $\lim_n \|x_n -
x\|=0$. Notice that $\tilde\rho \mathcal{P}({\ }^1 X)$ is
$sexp(B_X)$, the set of all strongly exposed points of $B_X$. Hence
we get the following corollary.

\begin{corollary}
Let $X$ be a real or complex Banach space. Then
\[ sexp(B_X)= \tilde\rho \mathcal{P}({\ }^1 X)\subset \tilde\rho \mathcal{P}({\ }^2 X)
\subset \dots \] In particular, when $X$ is a complex space,
\[\bigcup_{k=1}^\infty \tilde\rho \mathcal{P}({\ }^k X) \subset \rho
A_u(B_X) \subset \rho A_b(B_X) \subset ext_\mathbb{C}(B_X).\]
\end{corollary}
\begin{proof}
We may assume that $X\neq 0$ and need show that $\tilde\rho
\mathcal{P}({\ }^k X)\subset \tilde\rho \mathcal{P}({\ }^{k+1} X)$
for each $k\ge 1$. Suppose that $y\in \tilde\rho \mathcal{P}({\ }^k
X)$. So $y\in S_X$ and there is  $P\in  \mathcal{P}({\ }^k X)$ which
strongly attains its norm at $y$. Choose $y^*\in X^*$ such that
$\|y^*\|=y^*(y)=1$. Define $Q:X\to \mathbb{K}$ by $Q(x) =
y^*(x)P(x)$ for each $x\in X$. So $Q$ is a $(k+1)$-homogeneous
polynomial and it is easy to see that $Q$ strongly attains its norm
at $y$. The proof is complete.
\end{proof}

Recall that a Banach space $X$ is said to have the {\bf
Radon-Nikod\'ym property} if every nonempty bounded closed convex
subset in $X$ is a closed convex hull of its strongly exposed points
\cite{DU}. Reviewing the proof of Theorem~4.2 in \cite{CLS} (cf.
\cite{F}), we get the following.

\begin{proposition}\label{strongly}
Let $X, Y$ be Banach spaces over $\mathbb{K}$. Suppose that $X$ has
the Radon-Nikod\'ym property and $Y$ is nontrivial. Then for each
$P\in \mathcal{P}({\ }^k X:Y)$ and $\epsilon>0$, there is $Q\in
\mathcal{P}({\ }^k X:Y)$ such that $\|Q\|\le \epsilon$ and $P+Q$
strongly attains its norm.
\end{proposition}

Now we get the polynomial version of Bishop's theorem \cite{B,CLS}.
Recall that a subset $A$ of $B_X$ is said to be {\bf balanced } if
$\lambda\in S_\mathbb{K}$ implies  $\lambda A \subset A$.

\begin{proposition}Let $k$ be a positive integer and $X, Y$ Banach
spaces over $\mathbb{K}$. Suppose that $X$ has the Radon-Nikod\'ym
property and $Y$ is nontrivial. Then the set $\tilde\rho
\mathcal{P}({\ }^k X)$ is a norming subset of $\mathcal{P}({\ }^k
X:Y)$. In fact, the closure of $\tilde\rho \mathcal{P}({\ }^k X)$ is
the smallest closed balanced norming subset of $ \mathcal{P}({\ }^k
X:Y)$.
\end{proposition}
\begin{proof}By Proposition~\ref{strongly}, for each $P\in \mathcal{P}({\ }^k X:Y)$, there is a sequence
$\{P_n\}_{n=1}^\infty$ in $\mathcal{P}({\ }^k X:Y)$ such that
$\lim_n \|P_n-P\|=0$ and each $P_n$ strongly attains its norm at
$x_n$. So $\|P_n\|= \|P_n(x_n)\|$. Then for each $n\ge 1$,
\[ \|P_n(x_n)\|-\|P_n - P\|\le \|P(x_n)\|\le \|P_n(x_n)\|+\|P_n
-P\|\] holds. So $\lim_n \|P(x_n)\|=\lim_n \|P_n(x_n)\|= \lim_n
\|P_n\|= \|P\|.$ This shows that $\tilde\rho \mathcal{P}({\ }^k X)$
is a norming subset of $ \mathcal{P}({\ }^k X:Y)$. So it is clear
that the closure of $\tilde\rho \mathcal{P}({\ }^k X)$ is a closed
norming balanced subset of $S_X$. Suppose that $A$ is a closed
balanced norming subset of $ \mathcal{P}({\ }^k X:Y)$. Let $P$ be a
strongly norm-attaining element at $x_0$. Since $A$ is norming,
choose a sequence $\{x_n\}_{n=1}^\infty$ such that $\|P\|=\lim_n
\|P(x_n)\|$. Then we get a subsequence $\{y_l\}_{l=1}^\infty$ of
$\{x_n\}_{n=1}^\infty$ and $\lambda\in S_\mathbb{K}$ such that
$\lim_l y_l = \lambda x_0$. So $x_0$ is contained in $A$. This shows
that the closure of  $\tilde\rho \mathcal{P}({\ }^k X)$ is contained
in $A$. This completes the proof.
\end{proof}

The following Proposition~\ref{prop:indexonepeakpoint11} and
Theorem~\ref{prop:rnp1} are related with Problem~45 in \cite{KMP}.
Similar characterizations were shown in \cite{LMP,M,MG} for Banach
spaces with the numerical index 1. Recall that a point $x^*\in
B_{X^*}$ is called a {\bf weak-$*$ exposed point} of $B_{X^*}$ if
there is $x\in B_X$ such that $x^*(x)=1$ and $y^*(x)=1$ for some
$y^*\in B_{X^*}$ implies $y^*=x^*$. The corresponding point $x$ is
said to be a {\bf smooth point} of $S_X$. We denote by
$w^*exp(B_{X^*})$ the set of all weak-$*$ exposed points of
$B_{X^*}$. In the proof of Theorem~\ref{prop:rnp1}, we use the Mazur
theorem which says that if a Banach space $X$ is separable, then the
set of all smooth points of $S_X$ is dense in $S_X$.

\begin{proposition}\label{prop:indexonepeakpoint11}Let $k\ge 1$ be a
positive integer and $X$  a real or complex Banach space with
$n^{(k)}(X)=1$. Then $|x^*(x)|=1$ for each $x\in
\tilde\rho\mathcal{P}({\ }^k X)$ and $x^*\in ext(B_{X^*})$.
\end{proposition}
\begin{proof}
Suppose that $n^{(k)}(X)=1$. Fix $x_0^*\in ext(B_{X^*})$ and take
$P$ which strongly attains its norm at $x_0$ with $\|P\|=1$. We need
show that $|x_0^*(x_0)|=1$.

Fix $\epsilon>0$. Because $P$ strongly attains its norm at $x_0$,
there is $\alpha>0$ such that $|P(y)|\ge 1-\alpha$ for some $y\in
B_X$ implies $\inf_{\lambda\in S_\mathbb{K}}\|\lambda x - y\|\le
\epsilon$.

Set $F= \{ x^*\in B_{X^*} : |(x^*-x^*_0)(x_0)|\ge \epsilon\}$. Then
$F$ is weak-$*$ compact and $x^*_0$ is not contained in
$\overline{co}^{\,*}(F)$, the weak-$*$ closure of convex hull of
$F$. Otherwise, $x^*_0$ is an extreme point of
$\overline{co}^{\,*}(F)$ and $x_0^*\in F$ by the converse
Krein-Milman theorem. Hence there exist $y\in S_X$ and $\beta>0$
such that
\[ {\rm Re\ } x_0^*(y) > 1-\beta \ge {\rm Re\ } x^*(y), \ \ \
\forall x^*\in F.\] That is,  if ${\rm Re\ } x^*(y)>1-\beta$ for
some $x^*\in B_{X^*}$, then $|(x^*-x_0^*)(x_0)|<\epsilon$.

Let $f(x) = P(x)y$ for each $x\in X$. Then $f\in \mathcal{P}({\ }^k
X:X)$ and $v(f)= \|f\|=1$. Take $\delta = \min\{\alpha, \beta\}$.
Then there is $(x,x^*)\in \Pi(X)$ such that $|x^*f(x)|>1-\delta$.
This means that  $|P(x)|>1-\alpha$ and $|x^*(y)|= {\rm Re\ }
(\lambda_1 x^*)(y) > 1-\beta$ for some $\lambda_1\in S_\mathbb{K}$.
Hence $\inf_{\lambda\in S_\mathbb{K}} \|\lambda x_0 - x\| =
\|\lambda_2 x_0 - x\|\le \epsilon$ for some $\lambda_2\in
S_\mathbb{K}$ and $|(\lambda_1 x^*-x_0^*)(x_0)|<\epsilon$. By the
triangle inequality,
\begin{eqnarray*} \big| |x^*(x_0)| -|x^*(x)|\big| &\le&
|x^*(\lambda_2 x_0 - x)| \le \|\lambda_2 x_0 -x
\|\le \epsilon\ \ \mbox{and}\\
\big| |x^*(x_0)| - |x_0^*(x_0)| \big| &\le& |(\lambda_1
x^*-x_0^*)(x_0)| \ \le \epsilon. \end{eqnarray*} Notice that
$x^*(x)=1$. So $\big| 1- |x_0^*(x_0)| \big| \le 2\epsilon$.

Because $\epsilon>0$ is arbitrary, $|x_0^*(x_0)|=1$. This completes
the proof.
\end{proof}

\begin{theorem}\label{prop:rnp1}
Suppose that $k$ is a positive integer and $X$ is a nontrivial real
or complex Banach space with the Radon-Nikod\'ym property. The
following are equivalent.
\begin{enumerate}
\item \label{enu:numericalone31} $n^{(k)}(X)=1$.

\item \label{enu:numericalone41}  $|x^*(x)|=1$ for each $x\in \tilde\rho\mathcal{P}({\ }^k X)$ and $x^*\in
ext(B_{X^*})$.
\end{enumerate}
In addition, when $X$ is separable, $n^{(k)}(X)=1$ if and only if
$|x^*(x)|=1$ for each $x\in \tilde\rho\mathcal{P}({\ }^k X)$ and
$x^*\in w^*exp(B_{X^*})$.
\end{theorem}
\begin{proof}
$\ref{enu:numericalone31}\Rightarrow \ref{enu:numericalone41}$ is
proved by Proposition~\ref{prop:indexonepeakpoint11}. Conversely,
suppose that we have $\ref{enu:numericalone41}$. By
Proposition~\ref{strongly}, we have only to show that $v(f)=\|f\|$
for every strongly norm-attaining element $f\in \mathcal{P}({\ }^k
X:X)$. Fix a strongly norm-attaining element $f$ in $\mathcal{P}({\
}^k X:X)$ with $\|f\| = \|f(x)\|$. Then the set $F= \{y^*\in
B_{X^*}: y^*f(x)= \|f(x)\|\}$ is a nonempty weak-$*$ compact convex
subset of $B_{X^*}$. By the Krein-Milman theorem, there is an
extreme point $x^*$ of $F$ and it is easy to see that $x^*$ is also
an extreme point of $B_{X^*}$. Hence $|x^*(x)|=1$ by the assumption
and $v(f)\ge |x^*f(x)|= \|f(x)\|=\|f\|$. This proves
$\ref{enu:numericalone41}\Rightarrow \ref{enu:numericalone31}$.

Suppose that $X$ is separable. Because $w^*exp(B_{X^*})\subset
ext(B_{X^*})$, the ``only if" part is clear by
Proposition~\ref{prop:indexonepeakpoint11}. So we need prove the
sufficiency. By Proposition~\ref{strongly}, it is enough to show
that $v(f)=\|f\|$ for every strongly norm-attaining element $f\in
\mathcal{P}({\ }^k X:X)$. Fix a strongly norm-attaining element $f$
in $\mathcal{P}({\ }^k X:X)$ with $\|f\| = \|f(x)\|$. Since $X$ is
separable, the set of smooth points in $S_X$ is dense. Given
$\epsilon>0$, choose a smooth point $y\in S_X$ with
\[\left\|y-\frac{f(x)}{\|f\|}\right\|\le \frac{\epsilon}{\|f\|}.\]
Then there is $x^*\in w^*exp(B_{X^*})$ such that $x^*(y)=1$ and
$|x^*f(x)| \ge \|f\| -\epsilon$. So $|x^*(x)|=1$ by assumption.
Hence $v(f)\ge |x^*f(x)| \ge \|f\| -\epsilon$. Since $\epsilon>0$ is
arbitrary, we get $v(f)=\|f\|$. This completes the proof.
\end{proof}

Now we shall give a partial answer to Problem~44 in \cite{KMP},
which ask if the only real Banach space $X$ with $n^{(2)}(X)=1$ is
$X=\mathbb{R}$.

\begin{theorem}\label{thm2}
Suppose that $X$ is a real nontrivial finite dimensional Banach
space. Then $\tilde\rho \mathcal{P}({\ }^2 X)=S_X.$ In addition, if
$n^{(2)}(X)=1$ then $X=\mathbb{R}$.
\end{theorem}
\begin{proof} Suppose that $\dim X=n$ and identify $X$ with $X^{**}$. The first assertion is clear
when $n=1$. So we may assume that $n\ge 2$. Fix $x_1\in S_X$ and
choose $x_1^*\in S_{X^*}$ with $x_1^*(x_1)=1$. Then there is a
linearly independent subset $\{x_k^*\}_{k=2}^n$ in $B_{\ker x_1}$.
So $\{x_i^*\}_{i=1}^n$ is a basis of $X^*$.

Fix a positive sequence $\{\epsilon_k\}_{k=1}^n$ with $\sum_{k=1}^n
\epsilon_k <1/2$. Then define a 2-homogeneous polynomial \[P(x) =
x^*_1(x)^2 - \sum_{k=2}^n \epsilon_k x_k^*(x)^2.\] Notice that
$\|P\|= \|P(x_1)\|=1$. We shall show that $P$ strongly attains its
norm at $x_1$. Suppose that there is a sequence
$\{y_m\}_{m=1}^\infty$ in $B_X$ with $\lim_m |P(y_m)|=1$. This shows
that $\lim_m |x_1^*(y_m)|=1$ and $\lim_m |x_k^*(y_m)|=0$ for every
$k\ge 2$. Since $B_X$ is compact, we can choose a proper subsequence
of $\{y_m\}_m$ which converges to $\pm x_1$. This shows that $x_1$
is in $\tilde\rho \mathcal{P}({\ }^2 X)$. This proves the first
assertion.

In addition, if $n^{(2)}(X)=1$, then $|x^*(x)|=1$ for every $x\in
S_X=\tilde\rho\mathcal{P}({\ }^2X)$ and $x^*\in ext(B_{X^*})$ by
Proposition~\ref{prop:indexonepeakpoint11}. By the Minkowski
theorem, there is $x^*\in ext(B_{X^*})$. Then $|x^*(x)|=1$ for every
$x\in S_X$. If $\dim X\ge 2$, then there is $x\in S_X$ such that
$x^*(x)=0$. This is a contradiction.
\end{proof}

Now we denote by $n_u(X)$ the {\bf holomorphic numerical index}
defined by $n_u(X)= \inf \{ v(f): f\in A_u(B_X:X), \|f\|=1\}$.  The
proof of Proposition~\ref{prop:indexonepeakpoint11} gives the
following.
\begin{proposition}\label{prop:indexonepeakpoint1}
Let $X$ be a complex Banach space with $n_u(X)=1$. Then $|x^*(x)|=1$
for each $x\in \rho A_u(B_X)$ and $x^*\in ext(B_{X^*})$.
\end{proposition}

The proof of Theorem~\ref{prop:rnp1} gives the following if we use
Proposition~\ref{prop:indexonepeakpoint1} instead of
Proposition~\ref{prop:indexonepeakpoint11} and the fact that the set
of all strong peak functions in $A_u(B_X:X)$ is dense \cite{CLS}.
\begin{proposition}\label{prop:rnpa}
Suppose that $X$ is a nontrivial complex Banach space with the
Radon-Nikod\'ym property. The following are equivalent.
\begin{enumerate}
\item \label{enu:numericalone3} $n_u(X)=1$.

\item \label{enu:numericalone4} For each $x\in \rho A_u(B_X)$ and $x^*\in
ext(B_{X^*})$, we have $|x^*(x)|=1$.
\end{enumerate}
In addition, when $X$ is separable,  $n_u(X)=1$ if and only if
$|x^*(x)|=1$ for each $x\in \rho A_u(B_X)$ and $x^*\in
w^*exp(B_{X^*})$.
\end{proposition}

\begin{corollary}
Suppose that $X$ is a finite dimensional complex Banach space. The
following are equivalent.
\begin{enumerate}
\item \label{enu:numericalone1} $n_u(X)=1$.

\item \label{enu:numericalone2} For each $x\in ext_\mathbb{C}(B_X)$ and $x^*\in
ext(B_{X^*})$, we have $|x^*(x)|=1$.
\end{enumerate}
\end{corollary}
\begin{proof}
Notice that if $X$ is finite dimensional, $\rho A_u(B_X)=
ext_\mathbb{C}(B_X)$ (see \cite[Proposition~1.1]{CHL}). So
$\ref{enu:numericalone1}\iff\ref{enu:numericalone2}$ is shown by
Proposition~\ref{prop:rnpa}. This completes the proof.
\end{proof}

\subsection*{Acknowledgements} The author thanks Yun Sung Choi. He
introduced this topic to the author and kindly shared his ideas
which improved the previous version of this paper.

%
\end{document}